\documentclass[12pt,reqno]{amsart}
\usepackage{amsmath,amssymb,extarrows}
\usepackage{url}
\usepackage{tikz,enumerate}
\usepackage{diagbox}
\usepackage{appendix}
\usepackage{epic}
\usepackage{comment}

\newcommand{\padedvphantom}[3]{%
	\vtop{%
		\vbox{%
			\vspace*{#2}%
			\hbox{\vphantom{#1}}%
		}%
		\vspace*{#3}%
	}%
}

\usepackage{float} 

\usepackage{cite}

\usepackage{hyperref}
\usepackage{array}

\usepackage{booktabs}

\allowdisplaybreaks[4]

\usepackage{makecell} 
\usepackage{float} 
\usepackage{siunitx} 
\usepackage{multirow}
\usepackage[figuresright]{rotating}
\usepackage{stackengine}
\newcommand\xrowht[2][0]{\addstackgap[.5\dimexpr#2\relax]{\vphantom{#1}}}

\newtheorem{theorem}{Theorem}

\newtheorem{lemma}[theorem]{Lemma}

\theoremstyle{definition}

\theoremstyle{remark}

\numberwithin{equation}{section}
\numberwithin{theorem}{section}
\numberwithin{defn}{section}

\DeclareMathOperator{\CT}{CT}

\begin{document}
\title[Some new modular Nahm sums of ranks 3 and 4]
 {Some new modular Nahm sums of ranks 3 and 4}

\author{Zhineng Cao and Liuquan Wang}

\address[Z.\ Cao]{School of Mathematics and Statistics, Wuhan University, Wuhan 430072, Hubei, People's Republic of China}
\email{zhncao@whu.edu.cn}

\address[L.\ Wang]{School of Mathematics and Statistics, Wuhan University, Wuhan 430072, Hubei, People's Republic of China}

\email{wanglq@whu.edu.cn;mathlqwang@163.com}




\begin{abstract}
We discover six new families of modular Nahm sums in ranks 3 and 4. Two of them are rank three sums obtained by modifying two of Zagier's rank three examples. Three rank four families are derived by applying the lift‑dual operation to the rank three tadpole Nahm sums studied by Milas and Wang, while the other rank four family is found by the constant term method. To prove modularity, we establish Rogers--Ramanujan type identities that express these Nahm sums as  infinite products which are modular.
\end{abstract}

\maketitle

\section{Introduction and main results}\label{sec-intro}
For any positive integer $r$,  positive definite matrix $A\in \mathbb{Q}^{r\times r}$, vector $B\in \mathbb{Q}^r$ and scalar $C\in \mathbb{Q}$,  the rank $r$ Nahm sum
\begin{align}\label{eq-Nahm}
    f_{A,B,C}(q):=\sum_{n=(n_1,n_2,\dots,n_r)^\mathrm{T}\in \mathbb{Z}_{\geq 0}^r}\frac{q^{\frac{1}{2}n^\mathrm{T}An+B^\mathrm{T}n+C}}{(q;q)_{n_1}(q;q)_{n_2}\cdots (q;q)_{n_r}}
\end{align}
has drawn much attention in the theory of $q$-series, modular forms and rational conformal field theories. Here and throughout this paper we assume $|q|<1$ and use standard $q$-series notation:
\begin{align}
(a;q)_0:=1, \quad (a;q)_n:=\prod\limits_{k=0}^{n-1}(1-aq^k), \quad (a;q)_\infty :=\prod\limits_{k=0}^\infty (1-aq^k),  \\
(a_1,\dots,a_m;q)_n:=(a_1;q)_n\cdots (a_m;q)_n, \quad n\in \mathbb{N}\cup \{\infty\}.
\end{align}
As a key topic on Nahm sums, Nahm's problem asks for finding all triples $(A,B,C)$ such that $f_{A,B,C}(q)$ is modular, and such $(A,B,C)$ is referred to as a rank $r$ modular triple. Modular Nahm sums are usually characters of some 2-dimensional conformal field theories, and they also played important roles in the representation theory of Lie algebras.

Nahm's problem has been solved by Zagier in the rank one case, who proved that there are exactly seven rank one modular triples. However, when the rank $r\geq 2$, the problem is much more complicated and a complete classification of modular Nahm sums seems far from reach. Several works have been devoted to finding new modular triples. In particular, Zagier \cite{Zagier} listed 11 and 12 sets of modular triple candidates in the ranks 2 and 3 cases, respectively. The modularity of all of Zagier's candidates has been confirmed by the works of Zagier \cite{Zagier}, Vlasenko--Zwegers \cite{Vlasenko-Zwegers}, Cherednik--Feigin \cite{Feigin}, Wang \cite{Wang-rank2,Wang-rank3} and Cao--Rosengren--Wang \cite{CRW}.

Besides providing explicit modular triples, Zagier \cite[p.~50, (f)]{Zagier} also pointed out that there might exist some duality principles among modular Nahm sums. He suggests that for any rank $r$ modular triple $(A,B,C)$, its dual
$$\mathcal{D}(A,B,C)=(A^\star,B^\star,C^\star):=(A^{-1},A^{-1}B,\frac{1}{2}B^\mathrm{T}A^{-1}B-\frac{r}{24}-C)$$
might also be modular. This principle holds for many  known modular triples. However, Wang \cite{Wang-PAMS} recently found two modular rank four Nahm sums whose dual are not modular, and thereby providing counterexamples to this duality principle.

Apart from Zagier's examples in \cite{Zagier}, there are some other Nahm sums for which modularity remains an open problem. Calinescu, Milas and Penn \cite{CMP} conjectured that the rank $r$ tadpole Nahm sum associated with $(T_r,0,C)$ is modular for the tadpole Cartan matrix $T_r=(a_{ij})_{r\times r}$ and some suitable scalar $C$. Here $a_{rr}=1$, $a_{ii}=2$ for $1\leq i \leq r-1$, $a_{ij}=-1$ when $|i-j|=1$ and $a_{ij}=0$ otherwise. They proved this conjecture for $r=2$. The case $r=3$ was confirmed by Milas and Wang \cite{MW24}. In fact, they proved the modularity of the following triples $(A,B_i,C_i)$ ($i=1,2,\dots,6$)  (see \cite[Theorem 1.2]{MW24}):
\begin{align}\label{MW-mt}
\begin{split}
&A=T_3=\begin{pmatrix} 2 & -1 & 0 \\ -1 & 2 & -1 \\ 0 & -1 & 1 \end{pmatrix},~~ B_1=\begin{pmatrix} 0\\0\\0 \end{pmatrix},~~ B_2=\begin{pmatrix} 0\\0\\1/2 \end{pmatrix}, ~~ B_3=\begin{pmatrix} -1\\1\\0 \end{pmatrix}, \\
& B_4=\begin{pmatrix} 1\\-1\\1/2 \end{pmatrix}, ~~ B_5=\begin{pmatrix} -1\\1\\1/2 \end{pmatrix}, ~~ B_6=\begin{pmatrix} -2\\2\\-1/2 \end{pmatrix}, ~~ C_1=-\frac{7}{80}, ~~ C_2=\frac{1}{40},\\
& C_3=\frac{17}{80}, ~~ C_4=\frac{9}{40}, ~~ C_5=\frac{9}{40}, ~~ C_6=\frac{41}{40}.
\end{split}
\end{align}
The modularity of the tadpole Nahm sums in ranks $4$ and $5$ was recently proved by Shi and Wang \cite{Shi-Wang}.

It is known that one may construct higher rank Nahm sums from lower rank sums. One useful way to increase the rank of Nahm sums is the following formula (see e.g. \cite[p.\ 20]{Andrews1984}):
\begin{align}\label{Andrews-id}
\frac{1}{(q;q)_i(q;q)_j}=\sum_{k\geq 0} \frac{q^{(i-k)(j-k)}}{(q;q)_k(q;q)_{i-k}(q;q)_{j-k}}.
\end{align}
If we pick two variables $n_i$ and $n_j$ ($1\leq i<j\leq n$) and applying the above formula with $(i,j)$ replaced by $(n_i,n_j)$, then we can rewrite \eqref{eq-Nahm} as
\begin{align}\label{eq-Nahm-lift}
    f_{A,B,C}(q)=\sum_{(n_1,n_2,\dots,n_r,n_{r+1})^\mathrm{T}\in \mathbb{Z}_{\geq 0}^{r+1}}\frac{q^{\frac{1}{2}\widetilde{n}^\mathrm{T}A\widetilde{n}+B^\mathrm{T}\widetilde{n}+n_in_j+C}}{(q;q)_{n_1}(q;q)_{n_2}\cdots (q;q)_{n_r}(q;q)_{n_{r+1}}}
\end{align}
where
$$\widetilde{n}=(n_1,n_2,\dots,n_{i-1},n_i+n_{r+1},n_{i+1},\dots,n_{j-1},n_{j}+n_{r+1},n_{j+1},\dots,n_r,n_{r+1})^\mathrm{T}.$$
Motivated by this fact and Zagier's duality principle, the authors \cite{Cao-Wang2024,Cao-Wang2025} recently introduced the lift-dual operation to construct new modular triples from known examples. As a consequence, we \cite{Cao-Wang2024,Cao-Wang2025} found some new modular triples of ranks 3 and 4. Here we briefly review the lift-dual operation used in  \cite{Cao-Wang2025}. For any rank three modular triple $(A,B,C)$ with
\begin{align*}
&A=\begin{pmatrix} a_1 & a_2 & a_3 \\ a_2 & a_4 & a_5 \\ a_3 & a_5 & a_6 \end{pmatrix}, \quad B=\begin{pmatrix} b_1 \\ b_2 \\ b_3 \end{pmatrix} \quad \text{and} \quad C\in \mathbb{Q},
\end{align*}
we define three lifting operators as follows (see  \cite[(1.13)--(1.16)]{Cao-Wang2025}):
\begin{align}\label{eq-lift}
    \mathcal{L}_{i}: (A,B,C) \longmapsto (\widetilde{A}_i,\widetilde{B}_i,C), \quad i=1,2,3,
\end{align}
where
\begin{align}
&\widetilde{A}_1=\begin{pmatrix} a_1 & a_2 & a_3 & a_2+a_3 \\ a_2 & a_4 & a_5+1 & a_4+a_5 \\ a_3 & a_5+1 & a_6 & a_5+a_6 \\ a_2+a_3 & a_4+a_5 & a_5+a_6 & a_4+2a_5+a_6 \end{pmatrix}, \quad \widetilde{B}_1=\begin{pmatrix} b_1 \\ b_2 \\ b_3 \\ b_2+b_3\end{pmatrix}, \label{3-jk} \\
&\widetilde{A}_2=\begin{pmatrix} a_1 & a_2 & a_3+1 & a_1+a_3 \\ a_2 & a_4 & a_5 & a_2+a_5 \\ a_3+1 & a_5 & a_6 & a_3+a_6 \\ a_1+a_3 & a_2+a_5 & a_3+a_6 & a_1+2a_3+a_6 \end{pmatrix}, \quad \widetilde{B}_2=\begin{pmatrix} b_1 \\ b_2 \\ b_3 \\ b_1+b_3\end{pmatrix}, \label{3-ik} \\
&\widetilde{A}_3=\begin{pmatrix} a_1 & a_2+1 & a_3 & a_1+a_2 \\ a_2+1 & a_4 & a_5 & a_2+a_4 \\ a_3 & a_5 & a_6 & a_3+a_5 \\ a_1+a_2 & a_2+a_4 & a_3+a_5 & a_1+2a_2+a_4 \end{pmatrix}, \quad \widetilde{B}_3=\begin{pmatrix} b_1 \\ b_2 \\ b_3 \\ b_1+b_2\end{pmatrix}. \label{3-ij}
\end{align}
We use $\mathcal{L}_i(A)$ and $\mathcal{L}_i(B)$ to denote $\widetilde{A}_i$ and $\widetilde{B}_i$ in \eqref{eq-lift} and agree that $\mathcal{L}_i(C)=C$.

A key fact in the lifting process is the following equality:
\begin{align}
f_{A,B,C}(q)=f_{\widetilde{A}_i,\widetilde{B}_i,C}(q).
\end{align}
This is the $r=3$ case  of \eqref{eq-Nahm-lift}. As a consequence, if $(A,B,C)$ is a modular triple and $\widetilde{A}_i$ is positive definite, then $(\widetilde{A}_i,\widetilde{B}_i,C)$ is a modular triple as well.  Zagier's duality principle then inspires us to consider the modularity of the lift-dual triples $\mathcal{D}\mathcal{L}_i(A,B,C)$ ($i=1,2,3$). The authors \cite{Cao-Wang2025} applied such lift-dual operations to those rank three modular triples in \cite{Zagier,Cao-Wang2024} and obtained several new rank four modular triples.

This paper is a sequel to \cite{Cao-Wang2024,Cao-Wang2025}, and we aim to provide more modular Nahm sums in ranks 3 and 4. We discover such new examples in two different ways. The first way is proof-motivated and we obtain new modular triples when trying to simplify some series.  We discovered three new sets of  modular triples in this way. First, after a close examination of the first three examples in Zagier's rank three list \cite[Table 3]{Zagier}, we find some modular triples missing from it. We list the first and third examples in Zagier's rank three list  in Table \ref{tab-exam13}. Here $\alpha,\nu\in \mathbb{Q}$ and $h\in \mathbb{Z}$ can only take the values $0$ or $\pm 1$ to make $A$ positive definite. Inspired by these modular triples and the proof of their modularity in \cite{Zagier,Wang-rank3}, we find some new modular triples for matrices of the same form but with $h=\pm \frac{1}{2}$. We record these modular triples in Table \ref{tab-rank3-new}.
{\small
\begin{table}[htbp]
\centering
\caption{Zagier's Examples 1 and 3}
    \label{tab-exam13}
\begin{tabular}{|c|ccc|cc|}
\hline 
\padedvphantom{I}{3.5ex}{3.5ex}
    $A$   & \multicolumn{3}{c|}{$\left(\begin{smallmatrix}  \alpha h^2+1 & \alpha h & -\alpha h \\ \alpha h & \alpha & 1-\alpha \\ -\alpha h & 1-\alpha & \alpha \end{smallmatrix}\right)$ ($h=0,\pm1$)}  
    & \multicolumn{2}{c|}{$\left(\begin{smallmatrix}  \alpha h^2+1/2 & \alpha h & -\alpha h \\ \alpha h & \alpha & 1-\alpha \\ -\alpha h & 1-\alpha & \alpha \end{smallmatrix}\right)$ ($h=0$)}  \\
\hline 
\padedvphantom{I}{3.5ex}{3.5ex}
   $B$
     & $\left(\begin{smallmatrix}\alpha \nu h \\ \alpha \nu \\ -\alpha \nu \end{smallmatrix}\right)$  & $\left(\begin{smallmatrix} \alpha \nu h +{1}/{2} \\ \alpha \nu  \\ -\alpha \nu \end{smallmatrix}\right)$ & $\left(\begin{smallmatrix}
\alpha \nu h-1/2 \\ \alpha \nu \\ -\alpha \nu \end{smallmatrix}\right)$
     & $\left(\begin{smallmatrix} \alpha \nu h \\  \alpha \nu \\ -\alpha \nu  \end{smallmatrix}\right)$ & $\left(\begin{smallmatrix}  \alpha \nu h+1/2 \\  \alpha \nu  \\ -\alpha \nu \end{smallmatrix}\right)$ \\
     \padedvphantom{I}{1ex}{1ex}
   $C$
    & {\tiny $\frac{1}{2}\alpha \nu^2-\frac{1}{16}$} & {\tiny $\frac{1}{2}\alpha \nu^2$} &{\tiny $\frac{1}{2}\alpha\nu^2$} 
     & {\tiny $\frac{1}{2}\alpha {\nu^2}- \frac{1}{15}$} & {\tiny $ \frac{1}{2}\alpha \nu^2-\frac{1}{60}$}    \\
    \hline
\end{tabular}
\end{table}
}

{\small
\begin{table}[htbp]
\centering
\caption{Some new rank three modular triples}
    \label{tab-rank3-new}
\begin{tabular}{|c|cc|cc|}
\hline 
\padedvphantom{I}{3.5ex}{3.5ex}
    $A$   & \multicolumn{2}{c|}{$\left(\begin{smallmatrix}    \alpha h^2+1 & \alpha h & -\alpha h \\  \alpha h & \alpha & 1-\alpha \\  -\alpha h & 1-\alpha & \alpha \end{smallmatrix}\right)$ ($h=\pm \frac{1}{2}$)}  
    & \multicolumn{2}{c|}{$\left(\begin{smallmatrix}  \alpha h^2+1/2 & \alpha h & -\alpha h \\ \alpha h & \alpha & 1-\alpha \\ -\alpha h & 1-\alpha & \alpha \end{smallmatrix}\right)$ ($h=\pm \frac{1}{2}$)}  \\
\hline 
\padedvphantom{I}{3.5ex}{3.5ex}
   $B$
     & $\left(\begin{smallmatrix} \alpha \nu h \\ \alpha \nu \\ -\alpha \nu \end{smallmatrix}\right)$  & $\left(\begin{smallmatrix} \alpha \nu h-1/2 \\ \alpha \nu \\ -\alpha \nu \end{smallmatrix}\right)$ 
     & $\left(\begin{smallmatrix} \alpha \nu h \\  \alpha \nu \\ -\alpha \nu  \end{smallmatrix}\right)$ & $\left(\begin{smallmatrix}  \alpha \nu h+1/2 \\  \alpha \nu  \\ -\alpha \nu \end{smallmatrix}\right)$ \\
     \padedvphantom{I}{1ex}{1ex}
   $C$
    & {\tiny $\frac{1}{2}\alpha \nu^2-\frac{1}{16}$} & {\tiny $\frac{1}{2}\alpha \nu^2$} 
     & {\tiny $\frac{1}{2}\alpha {\nu^2}- \frac{1}{15}$} & {\tiny $ \frac{1}{2}\alpha \nu^2-\frac{1}{60}$}    \\
    \hline
\end{tabular}
\end{table}
}

We establish the following two theorems to justify the modularity of these new triples. To state our results in a more compact way, we denote
\begin{align*}
J_m=(q^{m};q^{m})_{\infty}, \quad 
J_{a,m}=(q^{a},q^{m-a},q^{m};q^{m})_{\infty},  \quad \overline{J}_{a,m}=(-q^{a},-q^{m-a},q^{m};q^{m})_{\infty}.
\end{align*}
Here $m$ is a positive integer and $0\leq a \leq m$. It is known that both $q^{m/24}J_{m}$ and $q^{m/24+mP_{2}(a/m)/2}J_{a,m}$ are modular forms of weight 1/2 where $P_{2}(x)=x^2-x+1/6$.
\begin{theorem}\label{thm-rank3}
Suppose that $h=\pm \frac{1}{2}$, $\alpha,\nu \in \mathbb{Q}$ and $\alpha>\frac{4}{7}$. We have
\begin{align}
&\sum_{i,j,k\geq 0} \frac{q^{\frac{1}{2}(\alpha h^2+1)i^2+\frac{1}{2}\alpha j^2+\frac{1}{2}\alpha k^2+\alpha h ij-\alpha h ik+(1-\alpha)jk+(\alpha \nu h-\frac{1}{2})i+\alpha \nu j-\alpha \nu k+\frac{1}{2}\alpha\nu^2}}{(q;q)_i(q;q)_j(q;q)_k}\nonumber \\
&=\frac{J_2}{J_1^2}\sum_{n \in \mathbb{Z}} q^{\frac{\alpha}{8}(n+2\nu)^2}, \label{exam1-1} \\
&\sum_{i,j,k\geq 0} \frac{q^{\frac{1}{2}(\alpha h^2+1)i^2+\frac{1}{2}\alpha j^2+\frac{1}{2}\alpha k^2+\alpha h ij-\alpha h ik+(1-\alpha)jk+\alpha \nu hi+\alpha \nu j-\alpha \nu k+\frac{1}{2}\alpha\nu^2}}{(q;q)_i(q;q)_j(q;q)_k}\nonumber \\
&=  \frac{J_{1,8}J_{6,16}}{J_1^2J_{16}}\sum_{n\in \mathbb{Z}+\nu} q^{\frac{\alpha}{2}n^2}+q^{\frac{1}{2}}\frac{J_{3,8}J_{2,16}}{J_1^2J_{16}}\sum_{n\in \mathbb{Z}+\nu+\frac{1}{2}} q^{\frac{\alpha}{2}n^2}. \label{exam1-2}
\end{align}    
\end{theorem}

\begin{theorem}\label{thm-rank3-2}
Suppose that $h=\pm \frac{1}{2}$, $\alpha,\nu\in \mathbb{Q}$ and $\alpha>\frac{2}{3}$. We have
\begin{align}
&\sum_{i,j,k\geq 0} \frac{q^{\frac{1}{2}(\alpha h^2+\frac{1}{2})i^2+\frac{1}{2}\alpha j^2+\frac{1}{2}\alpha k^2+\alpha h ij-\alpha h ik+(1-\alpha)jk+(\alpha \nu h)i+\alpha \nu j-\alpha \nu k+\frac{1}{2}\alpha\nu^2}}{(q;q)_i(q;q)_j(q;q)_k} \nonumber \\
&=\frac{J_{2,10}J_{6,20}}{J_1^2J_{20}} \sum_{n\in \mathbb{Z}+\nu} q^{\frac{\alpha}{2}n^2} +q^{\frac{1}{4}}\frac{J_{3,10}J_{4,20}}{J_1^2J_{20}} \sum_{n\in \mathbb{Z}+\nu+\frac{1}{2}} q^{\frac{\alpha}{2}n^2}, \label{add-thm-id-3}\\
&\sum_{i,j,k\geq 0} \frac{q^{\frac{1}{2}(\alpha h^2+\frac{1}{2})i^2+\frac{1}{2}\alpha j^2+\frac{1}{2}\alpha k^2+\alpha h ij-\alpha h ik+(1-\alpha)jk+(\alpha \nu h+\frac{1}{2})i+\alpha \nu j-\alpha \nu k+\frac{1}{2}\alpha\nu^2}}{(q;q)_i(q;q)_j(q;q)_k}\nonumber \\
&=\frac{J_{1,10}J_{8,20}}{J_1^2J_{20}}\sum_{n\in \mathbb{Z}+\nu} q^{\frac{\alpha}{2}n^2}  +q^{\frac{3}{4}} \frac{J_{4,10}J_{2,20}}{J_1^2J_{20}}\sum_{n\in \mathbb{Z}+\nu+\frac{1}{2}} q^{\frac{\alpha}{2}n^2}.  \label{add-thm-id-4}
\end{align}
\end{theorem}

Second, when analyzing certain sums through the constant term method, we find a new set of rank four modular triples  $(A,B,C)$  recorded in Table \ref{tab-new} where we list their duals as well.  We establish the following identities to justify the modularity of the triples $(A,B,C)$.
\begin{table}[htbp]
    \centering
    \footnotesize
    \caption{New rank four modular triples $(A,B,C)$ and their duals}\label{tab-new}
    \begin{tabular}{|c|ccc|c|ccc|}
      \hline \xrowht{50pt}
   $A$ & \multicolumn{3}{c|}{$\begin{pmatrix} 6 & -5 & 2 & -3 \\  -5 & 6 & -2 & 2 \\ 2 & -2 & 2 & -1 \\ -3 & 2 & -1 & 2  \end{pmatrix}$} & $\mathcal{D}(A)$ &  \multicolumn{3}{c|}{$\begin{pmatrix} 2 & 1 & 0 & 2 \\  1 & 4/5 & 1/5 & 4/5  \\ 0 & 1/5 & 4/5 & 1/5 \\ 2 & 4/5 & 1/5 & 14/5  \end{pmatrix}$} \\
   \hline \xrowht{50pt}
        $B$ & $\begin{pmatrix}  0 \\ 0 \\ 0 \\ 0  \end{pmatrix}$ & $\begin{pmatrix}  0 \\ 0 \\ -1/2 \\ 1/2  \end{pmatrix}$ & $\begin{pmatrix} 1 \\ -1 \\ 0 \\ 0 \end{pmatrix}$  & $\mathcal{D}(B)$ & $\begin{pmatrix} 0 \\ 0 \\ 0 \\ 0 \end{pmatrix}$ &
$\begin{pmatrix} 1 \\ 3/10 \\ -3/10 \\ 13/10 \end{pmatrix}$ &$\begin{pmatrix} 1 \\ 1/5 \\ -1/5 \\ 6/5 \end{pmatrix}$  \\
        $C$ & $-13/120$ & $17/480$ & $1/15$ & $\mathcal{D}(C)$ & $-7/120$ & $19/96$  &  $1/6$  \\
    \hline
    \end{tabular}
\end{table}
\begin{theorem}\label{thm-rank4-new} 
Let
\begin{align}
F(u,v,w,t;q^2):=\sum_{i,j,k,l\geq 0} \frac{u^iv^jw^kt^lq^{6i^2+6j^2+2k^2+2l^2-10ij+4ik-6il-4jk+4jl-2kl}}{(q^2;q^2)_i(q^2;q^2)_j(q^2;q^2)_k(q^2;q^2)_l}.  
\end{align}
We have
\begin{align}
&F(1,1,1,1;q^2)=\frac{J_{4}^{3}J_{8}^{3}J_{4,20}J_{12,40}}{J_{2}^{5}J_{16}^{2}J_{40}}+4q^2\frac{J_{8}J_{16}^{2}J_{6,20}J_{8,40}}{J_{2}^{3}J_{4}J_{40}}, \label{thm4-1} \\
&F(1,1,q^{-1},q;q^2)=\frac{\overline{J}_{1,4}\overline{J}_{3,8}J_{2,20}J_{16,40}}{J_{2}^{3}J_{40}}+q^2\frac{\overline{J}_{1,4}\overline{J}_{1,8}J_{8,20}J_{4,40}}{J_{2}^{3}J_{40}}, \label{thm4-2}\\
&F(q^2,q^{-2},1,1;q^2)=\frac{J_{4}^{5}\overline{J}_{2,8}J_{2,20}J_{16,40}}{J_{2}^{5}J_{8}^{2}J_{40}}+2q^2\frac{J_{8}^{2}\overline{J}_{2,8}J_{8,20}J_{4,40}}{J_{2}^{3}J_{4}J_{40}}.  \label{thm4-3}
\end{align}
\end{theorem}
We expect that the dual triples $\mathcal{D}(A,B,C)$ listed in Table \ref{tab-new} are also modular. However, we are not able to confirm their modularity at this stage and leave it to the interested readers.

The second way to discover new modular triples is utilizing the lift-dual operation introduced by the authors \cite{Cao-Wang2024,Cao-Wang2025}.  We find three new sets of rank four modular Nahm sums after applying the lift-dual operation to the modular triples in \eqref{MW-mt}, which was not considered in \cite{Cao-Wang2025}. Note that $\mathcal{L}_2(A)$ is a singular matrix and it remains to consider the lift-dual by $\mathcal{L}_1$ and $\mathcal{L}_3$. We find that both of them produce modular Nahm sums. 
For example, we list the lift and lift-dual via $\mathcal{L}_{3}$ of the modular triples in \eqref{MW-mt} in Table \ref{tab:MW-DL3}.
\begin{table}[htbp]
\centering
\caption{Lift and lift-dual via $\mathcal{L}_{3}$ of the modular triples $(A,B,C)$ in \eqref{MW-mt}}
\label{tab:MW-DL3}
\begin{tabular}{|c|cccccc|}
\hline \xrowht{40pt}
   $\mathcal{L}_{3}(A)$ & \multicolumn{6}{c|}{$\left(\begin{smallmatrix} 2 & 0 & 0 & 1 \\ 0 & 2 & -1 & 1 \\ 0 & -1 & 1 & -1 \\ 1 & 1 & -1 & 2 \end{smallmatrix}\right)$} \\
\hline \xrowht{34pt}
   $\mathcal{L}_{3}(B)$ & $\left(\begin{smallmatrix} 0 \\ 0 \\ 0 \\ 0 \end{smallmatrix}\right)$ & $\left(\begin{smallmatrix} 0 \\ 0 \\ 1/2 \\ 0 \end{smallmatrix}\right)$ & $\left(\begin{smallmatrix} -1 \\ 1 \\ 0 \\ 0 \end{smallmatrix}\right)$ & $\left(\begin{smallmatrix} 1 \\ -1 \\ 1/2 \\ 0 \end{smallmatrix}\right)$ & $\left(\begin{smallmatrix} -1 \\ 1 \\ -1/2 \\ 0 \end{smallmatrix}\right)$ & $\left(\begin{smallmatrix} -2 \\ 2 \\ -1/2 \\ 0 \end{smallmatrix}\right)$  \\ 
\xrowht{10pt}
   $\mathcal{L}_{3}(C)$ 
   & $-\frac{7}{80}$ & $\frac{1}{40}$ & $\frac{17}{80}$ & $\frac{9}{40}$ & $\frac{9}{40}$ & $\frac{41}{40}$ \\
\hline \xrowht{40pt}
   $\mathcal{DL}_{3}(A)$ & \multicolumn{6}{c|}{$\left(\begin{smallmatrix} 1 & 0 & -1 & -1 \\ 0 & 1 & 1 & 0 \\ -1 & 1 & 4 & 2 \\ -1 & 0 & 2 & 2 \end{smallmatrix}\right)$} \\
\hline \xrowht{34pt}
   $\mathcal{DL}_{3}(B)$ & $\left(\begin{smallmatrix} 0 \\ 0 \\ 0 \\ 0 \end{smallmatrix}\right)$ & $\left(\begin{smallmatrix} -1/2 \\ 1/2 \\ 2 \\ 1  \end{smallmatrix}\right)$ & $\left(\begin{smallmatrix} -1 \\ 1 \\ 2 \\ 1 \end{smallmatrix}\right)$ & $\left(\begin{smallmatrix} 1/2 \\ -1/2 \\ 0 \\ 0 \end{smallmatrix}\right)$ & $\left(\begin{smallmatrix} -1/2 \\ 1/2 \\ 0 \\ 0 \end{smallmatrix}\right)$ & $\left(\begin{smallmatrix} -3/2 \\ 3/2 \\ 2 \\ 1 \end{smallmatrix}\right)$\\
\xrowht{10pt}
   $\mathcal{DL}_{3}(C)$ 
    & $-\frac{19}{240}$ & $\frac{37}{120}$ & $\frac{149}{240}$ & $\frac{13}{120}$ & $\frac{13}{120}$ & $\frac{157}{120}$ \\
\hline
\end{tabular}
\end{table}
We establish the following identities to confirm the modularity of the triples $\mathcal{DL}_3(A,B,C)$ in Table \ref{tab:MW-DL3}. 
\begin{theorem}\label{thm-MW-L3}
Let  
\begin{align}
F(u,v,w,t;q):=\sum_{i,j,k,l\geq 0}\frac{u^{i}v^{j}w^{k}t^{l}q^{i^2/2+j^2/2+2k^2+l^2-ik-il+jk+2kl}}{(q;q)_{i}(q;q)_{j}(q;q)_{k}(q;q)_{l}}.   
\end{align}
We have 
\begin{align}
&F(1,1,1,1;q)=\frac{J_{1}^{5}J_{2,5}}{J_{1/2}^{3}J_{2}^{3}}, \label{thm-MW-id-1}\\
&F(q^{-1/2},q^{1/2},q^{2},q;q)=\frac{2J_{2}^{3}}{J_{1}^{4}}(\overline{J}_{7,20}-q\overline{J}_{3,20}), \label{thm-MW-id-2}\\
&F(q^{-1},q,q^{2},q;q)=q^{-1/2}\frac{J_{1}^{5}J_{1,5}}{J_{1/2}^{3}J_{2}^{3}}, \label{thm-MW-id-3} \\
&F(q^{1/2},q^{-1/2},1,1;q)=2\frac{J_{2}^{3}J_{2,5}}{J_{1}^{4}},  \label{thm-MW-id-4} \\
&F(q^{-1/2},q^{1/2},1,1;q)=4\frac{J_{2}^{3}}{J_{1}^{4}}(\overline{J}_{9,20}-q^{2}\overline{J}_{1,20}), \label{thm-MW-id-5}\\
&F(q^{-3/2},q^{3/2},q^{2},q;q)=4q^{-1}\frac{J_{2}^{3}J_{1,5}}{J_{1}^{4}}. \label{thm-MW-id-6}
\end{align}
\end{theorem}

The rest of this paper is organized as follows. In Section \ref{sec-pre} we collect some auxiliary identities and review some knowledge about Bailey pairs. We give proofs of Theorems \ref{thm-rank3}--\ref{thm-rank4-new} in Section \ref{sec-proof}. Section \ref{sec-MW} is devoted to the discussion of the  lift-dual of the Milas--Wang example \eqref{MW-mt}. We first show that the Nahm sums lifted from \eqref{MW-mt} through $\mathcal{L}_1$ can be generalized to include some parameters. We will then prove Theorems \ref{thm-MW-L3} and \ref{thm-MW-L1} which confirm the modularity of the lift-dual via $\mathcal{L}_3$ and $\mathcal{L}_1$ of \eqref{MW-mt}, respectively.

\section{Preliminaries}\label{sec-pre}
We first review some useful identities from the theory of $q$-series. 

The $q$-binomial theorem \cite[Theorem 2.1]{Andrews1998} states that
\begin{align}\label{q-binomial}
\sum_{n=0}^\infty \frac{(a;q)_n}{(q;q)_n}z^n=\frac{(az;q)_\infty}{(z;q)_\infty}. 
\end{align}
As its consequences, we have Euler's $q$-exponential identities \cite[Corollary 2.2]{Andrews1998}:
\begin{align}\label{Euler1}
\sum_{n=0}^\infty \frac{z^n}{(q;q)_n}&=\frac{1}{(z;q)_\infty}, \quad
|z|<1,  
\end{align}
\begin{align}\label{Euler2}
\sum_{n=0}^\infty \frac{z^nq^{\frac{n^2-n}{2}}}{(q;q)_n}&=(-z;q)_\infty.
\end{align}
The Jacobi triple product identity \cite[Theorem 2.8]{Andrews1998} asserts that
\begin{align}\label{Jacobi}
(z,q/z,q;q)_\infty=\sum_{n=-\infty}^{{\infty}} (-z)^n q^{\frac{n^2-n}{2}}.
\end{align}
We also need the $q$-Gauss summation formula \cite[Corollary 2.4]{Andrews1998}:
\begin{align} \label{Gauss}
\sum_{n=0}^\infty \frac{(a,b;q)_n}{(q,c;q)_n}\big(\frac{c}{ab}\big)^n=\frac{(c/a,c/b;q)_\infty}{(c,c/ab;q)_\infty},  \quad \left| \frac{c}{ab} \right|<1.
\end{align}

We recall some Rogers--Ramanujan type identities from Slater's list \cite{Slater1952}:
\begin{align}
&\sum_{n=0}^\infty \frac{q^{2n^2+n}}{(q;q)_{2n+1}}=\frac{J_{2}}{J_{1}} \quad \text{(S. 9)} \label{S. 9},\\
&\sum_{n=0}^\infty \frac{q^{2n^2+2n}}{(q;q)_{2n+1}}=\frac{J_{3,8}J_{2,16}}{J_{1}J_{16}} \quad \text{(S. 38)}, \label{S. 38}\\
&\sum_{n=0}^\infty \frac{q^{2n^2}}{(q;q)_{2n}}=\frac{J_{1,8}J_{6,16}}{J_{1}J_{16}} \quad \text{(S. 39)}, \label{S. 39}\\
&\sum_{n=0}^\infty \frac{q^{2n^2-n}}{(q;q)_{2n}}=\frac{J_{2}}{J_{1}} \quad \text{(S. 85)}, \label{S. 85}\\
&\sum_{n=0}^\infty \frac{q^{n^2+n}}{(q;q)_{2n+1}}=\frac{J_{3,10}J_{4,20}}{J_1J_{20}} \quad \text{(S. 94)}, \label{S. 94}\\
&\sum_{n=0}^\infty \frac{q^{n^2+2n}}{(q;q)_{2n+1}}=\frac{J_{4,10}J_{2,20}}{J_1J_{20}} \quad \text{(S. 96)}, \label{S. 96}\\
&\sum_{n=0}^\infty \frac{q^{n^2}}{(q;q)_{2n}}=\frac{J_{2,10}J_{6,20}}{J_1J_{20}} \quad \text{(S. 98)}, \label{S. 98}\\
&\sum_{n=0}^\infty \frac{q^{n^2+n}}{(q;q)_{2n}}=\frac{J_{1,10}J_{8,20}}{J_1J_{20}} \quad \text{(S. 99)}. \label{S. 99}
\end{align}
Here we use (S. $n$) to denote the $n$-th identity from Slater's list \cite{Slater1952}.

We will also need the theory of Bailey pairs. 
A pair of sequences $(\alpha_n(a;q),\beta_n(a;q))$ is called a Bailey pair relative to $a$ if for all $n\geq 0$, 
\begin{align}\label{defn-BP}
     \beta_n(a;q)=\sum_{k=0}^n\frac{\alpha_k(a;q)}{(q;q)_{n-k}(aq;q)_{n+k}}.
 \end{align}

We recall the E(1), E(4) and H(2) Bailey pairs from Slater's work \cite{Slater1951} (where $\alpha_{0}(1; q)=1$):
\begin{align}
&\left\{\begin{aligned}
&\alpha_{2n}(1;q)=2q^{4n^2},\quad \alpha_{2n+1}(1;q)=-2q^{4n^2+4n+1},\\
&\beta_{n}(1;q)=\frac{1}{(q^2;q^2)_{n}},
\end{aligned}\right.
\label{E1}
\tag{E(1)}\\[1.1ex]
&\left\{\begin{aligned}
&\alpha_{2n}(1;q)=q^{4n^2}(q^{2n}+q^{-2n}),\quad \alpha_{2n+1}(1;q)=-q^{4n^2+2n}-q^{4n^2+6n+2},\\
&\beta_{n}(1;q)=\frac{q^n}{(q^2;q^2)_{n}},
\end{aligned}\right.
\label{E4}
\tag{E(4)}\\[1.1ex]
&\left\{\begin{aligned}
&\alpha_{n}(1;q)=(-1)^nq^{n^2}(q^{n/2}+q^{-n/2}),\\
&\beta_{n}(1;q)=\frac{1}{(-q^{1/2},q;q)_{n}}.
\end{aligned}\right.
\label{H2}
\tag{H(2)}
\end{align}
Recall the following identity \cite[Eq.~(4.1)]{Slater1951}:
\begin{align}\label{Slater4.1}
\sum_{r=0}^{n}\frac{(1-aq^{2r})(-a)^rq^{(r^2+r)/2}(a,c,d;q)_{r}}{(a;q)_{n+r+1}(q;q)_{n-r}(q,aq/c,aq/d;q)_{r}(cd)^r}=\frac{(aq/cd;q)_{n}}{(q,aq/c,aq/d;q)_{n}}.   
\end{align}
We now choose some special values for $a$, $c$ and $d$ to obtain another three Bailey pairs:
\begin{enumerate}[(1)]
\item Letting $a=q$, $c=-q^{1/2}$, $d\rightarrow\infty$ in \eqref{Slater4.1}, we obtain
\begin{align}\label{Slater4.1-1}
\sum_{r=0}^{n}\frac{(1-q^{r+\frac{1}{2}})(-1)^rq^{r^2+\frac{r}{2}}}{(q^2;q)_{n+r}(q;q)_{n-r}(1-q^{1/2})}=\frac{1}{(-q^{3/2},q;q)_{n}}.      
\end{align}
Therefore, we have the following Bailey pair relative to $q$:
\begin{align}\label{BP-1}
\begin{split}
\alpha_{n}(q;q)&=\frac{(1-q^{n+\frac{1}{2}})(-1)^nq^{n^2+\frac{n}{2}}}{1-q^{1/2}},\\
\beta_{n}(q;q)&=\frac{1}{(-q^{3/2},q;q)_{n}}.
\end{split} 
\end{align}
\item Letting $a=1$, $c=-q$, $d\rightarrow\infty$ in \eqref{Slater4.1}, we obtain
\begin{align}\label{Slater4.1-2}
\frac{1}{(q;q)_{n}^{2}}+\frac{1}{2}\sum_{r=1}^{n}\frac{(1+q^{r})^2(-1)^rq^{r^2-r}}{(q;q)_{n+r}(q;q)_{n-r}}=\frac{1}{(-1,q;q)_{n}}.      
\end{align}
Therefore, we have the following Bailey pair relative to $1$:
\begin{align}\label{BP-2}
\begin{split}
\alpha_{0}(1;q)&=1,\quad \alpha_{n}(1;q)=\frac{1}{2}(1+q^{n})^2(-1)^nq^{n^2-n},\\
\beta_{n}(1;q)&=\frac{1}{(-1,q;q)_{n}}.
\end{split}  
\end{align}

\item Letting $a=q^2$, $c=-q$, $d\rightarrow\infty$ in \eqref{Slater4.1}, we obtain
\begin{align}\label{Slater4.1-3}
\sum_{r=0}^{n}\frac{(-1)^rq^{r^2+r}(1-q^{2r+2})(1+q)(1-q^{r+1})}{(q^2;q)_{n+r+1}(q;q)_{n-r}(1-q)(1+q^{r+1})}=\frac{1}{(-q^2,q;q)_{n}}.    
\end{align}
Note that 
$$1-q^{2r+2}=(1-q^{n+r+2})-q^{2r+2}(1-q^{n-r}).$$
We can rewrite \eqref{Slater4.1-3} as
\begin{align}
&\sum_{r=0}^{n}\frac{(-1)^rq^{r^2+r}(1+q)(1-q^{r+1})}{(q^2;q)_{n+r}(q;q)_{n-r}(1-q)(1+q^{r+1})} \nonumber\\
&-\sum_{r=0}^{n-1}\frac{(-1)^rq^{r^2+3r+2}(1+q)(1-q^{r+1})}{(q^2;q)_{n+r+1}(q;q)_{n-r-1}(1-q)(1+q^{r+1})} \nonumber\\
&=\frac{1}{(q^2,q;q)_{n}}+\frac{1+q}{1-q}\sum_{r=1}^{n}\frac{(-1)^rq^{r^2+r}}{(q^2;q)_{n+r}(q;q)_{n-r}}\Big(\frac{1-q^{r+1}}{1+q^{r+1}}+\frac{1-q^{r}}{1+q^{r}}\Big) \nonumber\\
&=\frac{1}{(q^2,q;q)_{n}}+\frac{1+q}{1-q}\sum_{r=1}^{n}\frac{(-1)^rq^{r^2+r}}{(q^2;q)_{n+r}(q;q)_{n-r}}\frac{2(1-q^{2r+1})}{(1+q^{r})(1+q^{r+1})} \nonumber\\
&=\frac{1}{(-q^2,q;q)_{n}}. 
\end{align}
Therefore, we have the following  Bailey pair relative to $q$:
\begin{align}\label{BP-3}
\begin{split}
\alpha_{0}(q;q)&=1,\quad \alpha_{n}(q;q)=\frac{2(1+q)(1-q^{2n+1})(-1)^nq^{n^2+n}}{(1-q)(1+q^n)(1+q^{n+1})},\\
\beta_{n}(q;q)&=\frac{1}{(-q^2,q;q)_{n}}.
\end{split} 
\end{align}
\end{enumerate} 
Bailey's lemma \cite{Bailey1949} plays a prominent role in the theory of Bailey pairs.
\begin{lemma}[Bailey's lemma]   
Suppose that $(\alpha_{n}(a;q),\beta_{n}(a;q))$ is a Bailey pair relative to $a$, then
\begin{align}\label{Bailey's lemma}
&\alpha_{n}'(a;q)=\frac{(\rho_{1},\rho_{2};q)_{n}}{(aq/\rho_{1},aq/\rho_{2};q)_{n}}\left( \frac{aq}{\rho_{1}\rho_{2}}\right)^{n}\alpha_{n}(a;q),  \\
&\beta_{n}'(a;q)=\sum_{k=0}^{n}\frac{(\rho_{1},\rho_{2};q)_{k}(aq/\rho_{1}\rho_{2};q)_{n-k}}{(aq/\rho_{1},aq/\rho_{2};q)_{n}(q;q)_{n-k}}\left( \frac{aq}{\rho_{1}\rho_{2}}\right)^{k}\beta_{k}(a;q)   
\end{align}
is another Bailey pair relative to $a$. In other words, if $(\alpha_{n}(a;q), \beta_{n}(a;q))$ is a Bailey pair, then we have
\begin{align}\label{Bailey's lemma-1}
&\frac{1}{(aq/\rho_{1},aq/\rho_{2};q)_{n}}\sum_{k=0}^{n}\frac{(\rho_{1},\rho_{2};q)_{k}(aq/\rho_{1}\rho_{2};q)_{n-k}}{(q;q)_{n-k}}\left( \frac{aq} {\rho_{1}\rho_{2}}\right)^{k}\beta_{k}(a;q) \nonumber\\
&=\sum_{r=0}^{n}\frac{(\rho_{1},\rho_{2};q)_{r}}{(q;q)_{n-r}(aq;q)_{n+r}(aq/\rho_{1},aq/\rho_{2};q)_{r}}\left( \frac{aq}{\rho_{1}\rho_{2}}\right)^{r}\alpha_{r}(a;q).
\end{align}
\end{lemma}

Letting $\rho_{1}\rightarrow \infty$ and $\rho_{2}\rightarrow \infty$, we obtain a new Bailey pair $(\alpha_n'(a;q),\beta_n'(a;q))$ from the Bailey pair $(\alpha_n(a;q),\beta_n(a;q))$ as follows:
\begin{align}
&\left\{\begin{aligned}\alpha_{n}'(a;q)&=a^nq^{n^2}\alpha_{n}(a;q),\\ 
\beta_{n}'(a;q)&=\sum_{k=0}^{n}\frac{a^kq^{k^2}}{(q;q)_{n-k}}\beta_{k}(a;q).
\end{aligned}\right.
\label{S1}
\tag{S1}
\end{align}
If we further let $n\rightarrow\infty$, then we deduce from \eqref{Bailey's lemma-1} that
\begin{align}
&\sum_{n=0}^{\infty}a^nq^{n^2}\beta_{n}(a;q)=\frac{1}{(aq;q)_{\infty}}\sum_{n=0}^{\infty}a^nq^{n^2}\alpha_{n}(a;q).
\label{S1*}
\tag{S1*}
\end{align}

Similarly, if we let $\rho_1\rightarrow \infty$, $\rho_{2}\in \{-\sqrt{aq},-q^{1/2},-aq^{1/2},-a^{1/2}q,-a^{1/2}\}$, we obtain the formulas \cite[(S2)--(S6)]{Bressoud2000} which are not needed and hence omitted here. If we further let $n\rightarrow \infty$ for the cases $\rho_{2}\in \{-\sqrt{aq},-aq^{1/2},-a^{1/2}q,-a^{1/2}\}$, we obtain the following identities:
\begin{align}
&\sum_{n=0}^{\infty}a^{n/2}q^{n^2/2}(-\sqrt{aq};q)_{n}\beta_{n}(a;q)=\frac{(-\sqrt{aq};q)_{\infty}}{(aq;q)_{\infty}}\sum_{n=0}^{\infty}a^{n/2}q^{n^2/2}\alpha_{n}(a;q),
\label{S2*}
\tag{S2*}\\[1.1ex]
&\sum_{n=0}^{\infty}q^{n^2/2}(-aq^{1/2};q)_{n}\beta_{n}(a;q)=\frac{(-q^{1/2};q)_{\infty}}{(aq;q)_{\infty}}\sum_{n=0}^{\infty}\frac{(-aq^{1/2};q)_{n}}{(-q^{1/2};q)_{n}}q^{n^2/2}\alpha_{n}(a;q), 
\label{S4*}
\tag{S4*}\\[1.1ex]
&\sum_{n=0}^{\infty}a^{n/2}q^{(n^2-n)/2}(-a^{1/2}q;q)_{n}\beta_{n}(a;q) \nonumber \\
&=\frac{(-a^{1/2};q)_{\infty}}{(aq;q)_{\infty}}\sum_{n=0}^{\infty}\frac{(-a^{1/2}q;q)_{n}}{(-a^{1/2};q)_{n}}a^{n/2}q^{(n^2-n)/2}\alpha_{n}(a;q), 
\label{S5*}
\tag{S5*}\\[1.1ex]
&\sum_{n=0}^{\infty}a^{n/2}q^{(n^2+n)/2}(-a^{1/2};q)_{n}\beta_{n}(a;q)\nonumber \\
&=\frac{(-a^{1/2}q;q)_{\infty}}{(aq;q)_{\infty}}\sum_{n=0}^{\infty}\frac{(-a^{1/2};q)_{n}}{(-a^{1/2}q;q)_{n}}a^{n/2}q^{(n^2+n)/2}\alpha_{n}(a;q).
\label{S6*}
\tag{S6*}
\end{align}

The following lemma allows us to change the parameter of Bailey pairs.
\begin{lemma}
If $(\alpha_{n}(a;q), \beta_{n}(a;q))$ is a Bailey pair relative to $a$, then   
\begin{align}\label{Lovejoy-a-aq-1}
\alpha_{n}'(aq;q)&=\frac{(1-aq^{2n+1})a^{n}q^{n^2}}{1-aq}\sum_{r=0}^{n}a^{-r}q^{-r^2}\alpha_{r}(a;q), \nonumber \\
\beta_{n}'(aq;q)&=\beta_{n}(a;q).   
\end{align}
is a Bailey pair relative to $aq$. Moreover,
\begin{align}\label{McLaughlin-a-a/q}
\alpha_{0}'(a/q;q)&=\alpha_{0}(a;q), \quad \alpha_{n}'(a/q;q)=(1-a)\left(\frac{\alpha_{n}(a;q)}{1-aq^{2n}}-\frac{aq^{2n-2}\alpha_{n-1}(a;q)}{1-aq^{2n-2}} \right), \nonumber\\
\beta_{n}'(a/q;q)&=\beta_{n}(a;q)    
\end{align}
is a Bailey pair relative to $a/q$.
\end{lemma}
The first assertion follows from the $b\rightarrow 0$ case of \cite[(2.4) and (2.5)]{Lovejoy2004}.
The second assertion can be found in \cite[(13.3)]{McLaughlin2018}.

\section{Proofs of Theorems \ref{thm-rank3}--\ref{thm-rank4-new}}\label{sec-proof}
We will follow the lines in \cite[pp.~48--50]{Zagier} and \cite[Proof of Lemma 4.1]{Wang-rank3} to give proofs of Theorems \ref{thm-rank3} and \ref{thm-rank3-2}.
\begin{proof}[Proof of Theorem \ref{thm-rank3}]
Suppose that $\alpha, \nu,h \in \mathbb{Q}$, $A_1>0$ and $B_1 \in \mathbb{R}$ so that the sum $S(q)$ defined below converges absolutely. Then we have
\begin{align}
&S(q):=\sum_{i,j,k\geq 0} \frac{q^{\frac{1}{2}(\alpha h^2+A_1)i^2+\frac{1}{2}\alpha j^2+\frac{1}{2}\alpha k^2+\alpha h ij-\alpha h ik+(1-\alpha)jk+(\alpha \nu h+B_1)i+\alpha \nu j-\alpha \nu k+\frac{1}{2}\alpha\nu^2}}{(q;q)_i(q;q)_j(q;q)_k}\nonumber \\
&=\sum_{i\geq 0} \frac{q^{\frac{A_1}{2}i^2+B_1i}}{(q;q)_i}\sum_{j,k\geq 0} \frac{q^{\frac{\alpha}{2}(hi+j-k+\nu)^2}\cdot q^{jk}}{(q;q)_j(q;q)_k} \nonumber \\
&=\sum_{i\geq 0}  \frac{q^{\frac{A_1}{2}i^2+B_1i}}{(q;q)_i} \sum_{n=-\infty}^\infty q^{\frac{\alpha}{2}(hi+n+\nu)^2} \sum_{j-k=n} \frac{q^{jk}}{(q;q)_j(q;q)_k} \nonumber \\
&=\frac{1}{(q;q)_\infty} \sum_{i\geq 0}\frac{q^{\frac{A_1}{2}i^2+B_1i}}{(q;q)_i} \sum_{n \in \mathbb{Z}} q^{\frac{\alpha}{2}(hi+n+\nu)^2}. \label{Sq-start}
\end{align}
Here for the last equality we used the Durfee rectangle identity (see e.g. \cite[(27)]{Zagier} or \cite[Corollary 2.6]{Andrews1998}): for any fixed integer $n$,
\begin{align}\label{Durfee}
\sum_{j-k=n} \frac{q^{jk}}{(q)_j(q)_k}=\frac{1}{(q;q)_\infty}.
\end{align}

When $h\in \frac{1}{2}+\mathbb{Z}$ we have
\begin{align}\label{Sq-reduce}
S(q)&=\frac{1}{(q;q)_\infty} \Big(\sum_{n\in \mathbb{Z}+\nu} q^{\frac{\alpha}{2}n^2}\sum_{i\geq 0} \frac{q^{2A_1i^2+2B_1i}}{(q;q)_{2i}} \nonumber \\
&\qquad +\sum_{n\in \mathbb{Z}+\nu+\frac{1}{2}} q^{\frac{\alpha}{2}n^2}\sum_{i\geq 0} \frac{q^{2A_1i^2+2(A_1+B_1)i+\frac{1}{2}A_1+B_1}}{(q;q)_{2i+1}} \Big).
\end{align}

When $(A_1,B_1)=(1,-1/2)$, using \eqref{S. 85} and \eqref{S. 9} we deduce that
\begin{align}
S(q)=\frac{J_2}{J_1^2} \Big(\sum_{n\in \mathbb{Z}+\nu} q^{\frac{\alpha}{2}n^2} +\sum_{n\in \mathbb{Z}+\nu+\frac{1}{2}} q^{\frac{\alpha}{2}n^2}\Big)=\frac{J_2}{J_1^2}\sum_{n\in \mathbb{Z}} q^{\frac{\alpha}{8}(n+2\nu)^2}.
\end{align}
This proves \eqref{exam1-1}.

Similarly, when $(A_1,B_1)=(1,0)$, using  \eqref{S. 39} and \eqref{S. 38}  in \eqref{Sq-start} we obtain \eqref{exam1-2}.

It remains to check the convergence conditions. One sufficient condition to ensure that $S(q)$ converges absolutely is to require the matrix 
$$A=\begin{pmatrix}  \alpha h^2+1 & \alpha h & -\alpha h \\  \alpha h & \alpha & 1-\alpha \\  -\alpha h & 1-\alpha & \alpha \end{pmatrix}$$
to be positive definite. This then requires $\det A=(2-h^2)\alpha-1>0$ and $\alpha>0$. Hence $2-h^2>0$ and we must have $h=\pm \frac{1}{2}$ and $\alpha>\frac{1}{2-h^2}=4/7$. It is easy to check that $A$ is positive definite  under such conditions.
\end{proof}

\begin{proof}[Proof of Theorem \ref{thm-rank3-2}]
We first assume that $h\in \frac{1}{2}+\mathbb{Z}$. When $(A_1,B_1)=(1/2,0)$, using \eqref{S. 98} and \eqref{S. 94} in \eqref{Sq-reduce} we obtain \eqref{add-thm-id-3}.

When $(A_1,B_1)=(1/2,1/2)$, using \eqref{S. 99} and \eqref{S. 96} in \eqref{Sq-reduce} we obtain \eqref{add-thm-id-4}.

As before, to ensure that $S(q)$ converges absolutely, we consider the matrix associated with the Nahm sum:
\begin{align}\label{rank3-A-1}
A=\begin{pmatrix}  \alpha h^2+\frac{1}{2} & \alpha h & -\alpha h \\  \alpha h & \alpha & 1-\alpha \\  -\alpha h & 1-\alpha & \alpha \end{pmatrix}.
\end{align}
To make it positive definite, we must have $\det A=\alpha(1-h^2)-\frac{1}{2}>0$ and $\alpha>0$. Hence  $|h|<1$ and $\alpha>\frac{1}{2(1-h^2)}$. Therefore, we have $h=\pm \frac{1}{2}$ and $\alpha>\frac{2}{3}$. It is easy to check that $A$ is positive definite under such conditions. 
\end{proof}

We now present a proof for the modularity of the modular triples $(A,B,C)$ in Table \ref{tab-new}.  For any series $f(z_1,\dots,z_k)=\sum_{n_1,n_2,\dots,n_k\in \mathbb{Z}}a_{n_1n_2\dots n_k}z_1^{n_1}z_2^{n_2}\cdots z_k^{n_k}$  in the variables $z_1,z_2,\dots,z_k$,  we define the constant term extracting operator $\mathrm{CT}$ as follows:
\begin{align}
    \mathrm{CT}_{z_1,\dots,z_n} f(z_1,z_2,\dots,z_n):=a_{00\dots,0}.
\end{align}
We will utilize the constant term method, which has been used frequently to prove Nahm sum identities in works such as \cite{Cao-Wang2024,Cao-Wang2025,MW24,Shi-Wang,Wang-rank2,Wang-rank3}. It starts by expressing the Nahm sum in consideration as the constant term of certain multi-sums, and then we need to transform these sums to different forms (such as infinite products). By some delicate manipulations, in certain cases we may eventually extract the constant term and obtain the desired identity for the original Nahm sum.

Compared with the other proofs in this paper, the proof of Theorem~\ref{thm-rank4-new} is quite tricky. We first express the Nahm sum as a constant term with respect to three variables. Then we simplify it using \eqref{q-binomial}, \eqref{Jacobi} and \eqref{Gauss}.
\begin{proof}[Proof of Theorem \ref{thm-rank4-new}]
By (\ref{Euler1}) and (\ref{Jacobi}) we have
\begin{align}\label{new-1-1}
&F(u,v,w,t;q^2)=
\sum_{i,j,k,l\geq 0} \frac{u^iv^jw^kt^lq^{(2i-2j+k-l)^2+(i-j)^2+(i-l)^2+j^2+k^2}}{(q^2;q^2)_i(q^2;q^2)_j(q^2;q^2)_k(q^2;q^2)_l} \nonumber\\ 
&=\CT_{x}\CT_{y}\CT_{z} \sum_{i\geq 0}\frac{(ux^2yz)^i}{(q^2;q^2)_i}\sum_{j\geq 0}\frac{(v/(x^2y))^jq^{j^2}}{(q^2;q^2)_j}\sum_{k\geq 0}\frac{(wx)^kq^{k^2}}{(q^2;q^2)_k}\sum_{l\geq 0}\frac{(t/(xz))^l}{(q^2;q^2)_l} \nonumber\\ 
&\quad \times \sum_{m=-\infty}^{\infty}x^{-m}q^{m^2}\sum_{n=-\infty}^{\infty}y^{-n}q^{n^2}\sum_{s=-\infty}^{\infty}z^{-s}q^{s^2} \nonumber\\ 
&=\CT_{x}\CT_{y} (-qv/(x^2y),-qwx;q^2)_{\infty}(-qx,-q/x,q^2;q^2)_{\infty}(-qy,-q/y,q^2;q^2)_{\infty} \nonumber\\ 
&\quad\times \CT_{z} \frac{(-qz,-q/z,q^2;q^2)_{\infty}}{(ux^2yz,t/(xz);q^2)_{\infty}} \nonumber\\ 
&=\CT_{x}\CT_{y} (-qv/(x^2y),-qwx;q^2)_{\infty}(-qx,-q/x,q^2;q^2)_{\infty}(-qy,-q/y,q^2;q^2)_{\infty} \nonumber\\
&\quad\times (q^2;q^2)_{\infty}\CT_{z} \sum_{i\geq 0}\frac{(-q/(ux^2y);q^2)_i}{(q^2;q^2)_i}(ux^2yz)^i \nonumber \\
&\quad \times \sum_{j\geq 0}\frac{(-qx/t;q^2)_j}{(q^2;q^2)_j}(t/(xz))^j\quad \text{(by (\ref{q-binomial}))} \nonumber\\
&=\CT_{x}\CT_{y} (-qv/(x^2y),-qwx;q^2)_{\infty}(-qx,-q/x,q^2;q^2)_{\infty}(-qy,-q/y,q^2;q^2)_{\infty} \nonumber\\
&\quad\times (q^2;q^2)_{\infty}\sum_{i\geq 0}\frac{(-q/(ux^2y),-qx/t;q^2)_i}{(q^2,q^2;q^2)_i}(utxy)^i \nonumber\\
&=\CT_{x}\CT_{y} (-qv/(x^2y),-qwx;q^2)_{\infty}(-qx,-q/x,q^2;q^2)_{\infty}(-qy,-q/y,q^2;q^2)_{\infty} \nonumber\\
&\quad\times \frac{(-qux^2y,-qt/x;q^2)_{\infty}}{(utxy;q^2)_{\infty}} \quad \text{(by (\ref{Gauss}))} \nonumber\\
&=\CT_{x}\CT_{y} (-qwx,-qt/x,-qx,-q/x,q^2;q^2)_{\infty} \nonumber \\
&\qquad \times \frac{(-qux^2y,-qv/(x^2y),-qy,-q/y,q^2;q^2)_{\infty}}{(utxy;q^2)_{\infty}}.
\end{align}

Letting $v=1/u$ in \eqref{new-1-1}, we have
\begin{align}\label{new-1-2}
&F(u,1/u,w,t;q^2)= \CT_{x}(-qwx,-qt/x,-qx,-q/x,q^2;q^2)_{\infty} \nonumber\\ 
&\quad\times \CT_{y}\frac{(-qux^2y,-q/(ux^2y),-qy,-q/y,q^2;q^2)_{\infty}}{(utxy;q^2)_{\infty}} \nonumber\\ 
&=\CT_{x} (-qwx,-qt/x,-qx,-q/x;q^2)_{\infty} \nonumber\\ 
&\quad\times \CT_{y} \sum_{i\geq 0}\frac{(utxy)^i}{(q^2;q^2)_i}\sum_{m=-\infty}^{\infty}(ux^2y)^{-m}q^{m^2}\sum_{n=-\infty}^{\infty}y^{-n}q^{n^2}
\nonumber\\ 
&=\CT_{x} (-qwx,-qt/x,-qx,-q/x;q^2)_{\infty} \nonumber\\ 
&\quad\times \sum_{i\geq 0}\frac{(utx)^i}{(q^2;q^2)_i}\sum_{m=-\infty}^{\infty}(ux^2)^{-m}q^{2m^2+i^2-2im} \nonumber\\
&=\CT_{x} (-qwx,-qt/x,-qx,-q/x;q^2)_{\infty} \Big( \sum_{i\geq 0}\frac{u^it^{2i}q^{2i^2}}{(q^2;q^2)_{2i}}\sum_{m=-\infty}^{\infty}(ux^2)^{i-m}q^{2(m-i)^2} \nonumber \\
&\qquad +\sum_{i\geq 0}\frac{u^{i+1}t^{2i+1}q^{2i^2+2i+1}}{(q^2;q^2)_{2i+1}}\sum_{m=-\infty}^{\infty}x(ux^2)^{i-m}q^{2(m-i)^2+2(i-m)}\Big) \nonumber\\ 
&=\CT_{x} (-qwx,-qt/x,-qx,-q/x;q^2)_{\infty} \Big( \sum_{i\geq 0}\frac{u^it^{2i}q^{2i^2}}{(q^2;q^2)_{2i}}\sum_{m=-\infty}^{\infty}(ux^2)^{m}q^{2m^2}  \nonumber \\
&\qquad +\sum_{i\geq 0}\frac{u^{i+1}t^{2i+1}q^{2i^2+2i+1}}{(q^2;q^2)_{2i+1}}\sum_{m=-\infty}^{\infty}u^mx^{2m+1}q^{2m^2+2m}\Big). 
\end{align}
Here for the penultimate line we split the sum into two parts according to the parity of $i$.

Letting $w=1/t$ in \eqref{new-1-2}, we have
\begin{align}\label{new-1-3}
&F(u,1/u,1/t,t;q^2)\nonumber\\ 
&=\CT_{x} (-qx/t,-qt/x,-qx,-q/x;q^2)_{\infty} \times\Big( \sum_{i\geq 0}\frac{u^it^{2i}q^{2i^2}}{(q^2;q^2)_{2i}}\sum_{m=-\infty}^{\infty}(ux^2)^{m}q^{2m^2}  \nonumber\\ 
&\quad +\sum_{i\geq 0}\frac{u^{i+1}t^{2i+1}q^{2i^2+2i+1}}{(q^2;q^2)_{2i+1}}\sum_{m=-\infty}^{\infty}u^mx^{2m+1}q^{2m^2+2m}\Big) \nonumber\\ 
&=\frac{1}{(q^2;q^2)_{\infty}^{2}}\sum_{i\geq 0}\frac{u^it^{2i}q^{2i^2}}{(q^2;q^2)_{2i}}\CT_{x}\sum_{m,n,s=-\infty}^{\infty}(ux^2)^{m}q^{2m^2}(x/t)^{-n}q^{n^2}x^{-s}q^{s^2} \nonumber\\ 
&\quad+\frac{1}{(q^2;q^2)_{\infty}^{2}}\sum_{i\geq 0}\frac{u^{i+1}t^{2i+1}q^{2i^2+2i+1}}{(q^2;q^2)_{2i+1}}\CT_{x}\sum_{m,n,s=-\infty}^{\infty}u^mx^{2m+1}q^{2m^2+2m}(x/t)^{-n}q^{n^2}x^{-s}q^{s^2} \nonumber\\ 
&=\frac{1}{(q^2;q^2)_{\infty}^{2}}\sum_{i\geq 0}\frac{u^it^{2i}q^{2i^2}}{(q^2;q^2)_{2i}}\sum_{m,n=-\infty}^{\infty}u^{m}t^{n}q^{2(m-n)^2+4m^2} \nonumber\\ 
&\quad+\frac{1}{(q^2;q^2)_{\infty}^{2}}\sum_{i\geq 0}\frac{u^{i+1}t^{2i+1}q^{2i^2+2i+1}}{(q^2;q^2)_{2i+1}}\sum_{m,n=-\infty}^{\infty}u^{m}t^{n}q^{2(m-n)^2+2(m-n)+4m^2+4m+1}  \nonumber\\ 
&=\frac{1}{(q^2;q^2)_{\infty}^{2}}\sum_{i\geq 0}\frac{u^it^{2i}q^{2i^2}}{(q^2;q^2)_{2i}}\sum_{m=-\infty}^{\infty}(ut)^{m}q^{4m^2}\sum_{n=-\infty}^{\infty}t^{-n}q^{2n^2} \nonumber\\ 
&\quad+\frac{1}{(q^2;q^2)_{\infty}^{2}}\sum_{i\geq 0}\frac{u^{i+1}t^{2i+1}q^{2i^2+2i+1}}{(q^2;q^2)_{2i+1}}\sum_{m=-\infty}^{\infty}(ut)^{m}q^{4m^2+4m+1}\sum_{n=-\infty}^{\infty}t^{-n}q^{2n^2+2n} \nonumber\\ 
&=\frac{(-q^4ut,-q^4/(ut),q^8;q^8)_{\infty}(-q^2t,-q^2/t,q^4;q^4)_{\infty}}{(q^2;q^2)_{\infty}^{2}}\sum_{i\geq 0}\frac{u^it^{2i}q^{2i^2}}{(q^2;q^2)_{2i}} \nonumber\\ 
&\quad+\frac{(-q^8ut,-1/(ut),q^8;q^8)_{\infty}(-q^4/t,-t,q^4;q^4)_{\infty}}{(q^2;q^2)_{\infty}^{2}}\sum_{i\geq 0}\frac{u^{i+1}t^{2i+1}q^{2i^2+2i+2}}{(q^2;q^2)_{2i+1}}.
\end{align}

(1) Setting $(u,t)=(1,1)$ in \eqref{new-1-3},  we have
\begin{align}
&F(1,1,1,1;q^2)=\frac{(-q^4,-q^4,q^8;q^8)_{\infty}(-q^2,-q^2,q^4;q^4)_{\infty}}{(q^2;q^2)_{\infty}^{2}}\sum_{i\geq 0}\frac{q^{2i^2}}{(q^2;q^2)_{2i}} \nonumber\\
&\quad +\frac{(-q^8,-1,q^8;q^8)_{\infty}(-q^4,-1,q^4;q^4)_{\infty}}{(q^2;q^2)_{\infty}^{2}}\sum_{i\geq 0}\frac{q^{2i^2+2i+2}}{(q^2;q^2)_{2i+1}}.
\end{align}
Substituting \eqref{S. 98}  and \eqref{S. 94} into it, we obtain \eqref{thm4-1}.

(2)  Setting $(u,t)=(1,q)$ in \eqref{new-1-3},  we have
\begin{align}
&F(1,1,q^{-1},q;q^2)=\frac{(-q^5,-q^3,q^8;q^8)_{\infty}(-q^3,-q,q^4;q^4)_{\infty}}{(q^2;q^2)_{\infty}^{2}}\sum_{i\geq 0}\frac{q^{2i^2+2i}}{(q^2;q^2)_{2i}} \nonumber\\
&\quad +\frac{(-q^9,-1/q,q^8;q^8)_{\infty}(-q^3,-q,q^4;q^4)_{\infty}}{(q^2;q^2)_{\infty}^{2}}\sum_{i\geq 0}\frac{q^{2i^2+4i+3}}{(q^2;q^2)_{2i+1}}. 
\end{align}
Substituting \eqref{S. 99}  and \eqref{S. 96} into it, we obtain \eqref{thm4-2}.

(3) Setting $(u,t)=(q^2,1)$ in  \eqref{new-1-3}, we have
\begin{align}
&F(q^2,q^{-2},1,1;q^2)=\frac{(-q^6,-q^2,q^8;q^8)_{\infty}(-q^2,-q^2,q^4;q^4)_{\infty}}{(q^2;q^2)_{\infty}^{2}}\sum_{i\geq 0}\frac{q^{2i^2+2i}}{(q^2;q^2)_{2i}} \nonumber\\
&\quad +\frac{(-q^{10},-1/q^2,q^8;q^8)_{\infty}(-q^4,-1,q^4;q^4)_{\infty}}{(q^2;q^2)_{\infty}^{2}}\sum_{i\geq 0}\frac{q^{2i^2+4i+4}}{(q^2;q^2)_{2i+1}}. 
\end{align}
Substituting \eqref{S. 99} and \eqref{S. 96}  into it, we obtain \eqref{thm4-3}.
\end{proof}

\section{Lift-dual of the Milas--Wang example}\label{sec-MW}
As mentioned in Section \ref{sec-intro}, only the lifts of \eqref{MW-mt} through $\mathcal{L}_1$ and $\mathcal{L}_3$  are meaningful. Surprisingly, we find that the modular triples lifted from \eqref{MW-mt} using $\mathcal{L}_1$ can be combined and extended to two general triples with parameters. We list these general triples and their duals in  Table \ref{tab:MW-DL1} where $\nu\in\mathbb{Q}$ and $c\in\{-1/2,0,1/2\}$.

\begin{table}[htbp]
\centering
\caption{The parameterized lift and lift-dual via $\mathcal{L}_{1}$ of $(A,B,C)$ in \eqref{MW-mt}}
\label{tab:MW-DL1}
\begin{tabular}{|c|cc|c|cc|}
\hline \xrowht{40pt}
   $\mathcal{L}_1(A)$
     &  \multicolumn{2}{c|}{$\left(\begin{smallmatrix} 2 & -1 & 0 & -1 \\ -1 & 2 & 0 & 1 \\ 0 & 0 & 1 & 0 \\ -1 & 1 & 0 & 1 \end{smallmatrix}\right)$}
   &$\mathcal{DL}_1(A)$
      & \multicolumn{2}{c|}{$\left(\begin{smallmatrix} 1 & 0 & 0 & 1 \\ 0 & 1 & 0 & -1 \\ 0 & 0 & 1 & 0 \\ 1 & -1 & 0 & 3 \end{smallmatrix}\right)$} \\
\hline \xrowht{34pt}
   $\mathcal{L}_1(B)$
     &  $\left(\begin{smallmatrix} -2\nu \\ 2\nu \\ c \\ \nu 
     \end{smallmatrix}\right)$ & $\left(\begin{smallmatrix} -2\nu \\ 2\nu \\ c \\ \nu+1/2 \end{smallmatrix}\right)$ 
    & $\mathcal{DL}_1(B)$
      & $\left(\begin{smallmatrix} \nu \\ -\nu \\ c \\ \nu  \end{smallmatrix}\right)$ & $\left(\begin{smallmatrix} \nu \\ -\nu \\ c \\ \nu+1 \end{smallmatrix}\right)$\\
\xrowht{12pt}
     $\mathcal{L}_1(C)$
     & $\frac{|c|}{8}-\frac{7}{80}+\nu^2$ & $\frac{|c|}{8}-\frac{3}{80}+\nu^2$
   & $\mathcal{DL}_1(C)$
      & $\frac{|c|}{8}-\frac{19}{240}+\frac{\nu^2}{2}$ & $\frac{|c|}{8}+\frac{29}{240}+\frac{\nu^2}{2}$ \\
    \hline
\end{tabular}
\end{table}

We establish the following identities to confirm the modularity of the triples $\mathcal{L}_1(A,B,C)$ in Table \ref{tab:MW-DL1}.
\begin{theorem}\label{thm-MW-L1}
Let
\begin{align}
F(u,v,w,t;q):=\sum_{i,j,k,l\geq 0}\frac{u^{i}v^{j}w^{k}t^{l}q^{i^2+j^2+\frac{1}{2}k^2+\frac{1}{2}l^2-ij-il+jl}}{(q;q)_{i}(q;q)_{j}(q;q)_{k}(q;q)_{l}}.   
\end{align}
We have
\begin{align}
F(q^{-2\nu},q^{2\nu},q^{c},q^{\nu};q)&=(-q^{c+\frac{1}{2}};q)_{\infty}\Big(\frac{J_{2,10}J_{6,20}\overline{J}_{1-2\nu,2}}{J_{1}^{2}J_{20}}+q^{\frac{1}{2}-\nu}\frac{J_{3,10}J_{4,20}\overline{J}_{2\nu,2}}{J_{1}^{2}J_{20}}\Big), \label{Dex6-L1-1} \\
F(q^{-2\nu},q^{2\nu},q^{c},q^{\nu+\frac{1}{2}};q)&=(-q^{c+\frac{1}{2}};q)_{\infty}\Big(\frac{J_{1,10}J_{8,20}\overline{J}_{1-2\nu,2}}{J_{1}^{2}J_{20}}+q^{1-\nu}\frac{J_{4,10}J_{2,20}\overline{J}_{2\nu,2}}{J_{1}^{2}J_{20}}\Big). \label{Dex6-L1-2}
\end{align}
\end{theorem}
\begin{proof}
Summing over $k$ using \eqref{Euler2}, we have
\begin{align}\label{Dex6-L1}
F(u,v,w,t;q)=(-wq^{1/2};q)_{\infty}\sum_{i,j,l\geq 0}\frac{u^{i}v^{j}t^{l}q^{i^2+j^2+\frac{1}{2}l^2-ij-il+jl}}{(q;q)_{i}(q;q)_{j}(q;q)_{l}}.    
\end{align}
Substituting \eqref{add-thm-id-3} and \eqref{add-thm-id-4} with $\alpha=2$ and $h=1/2$ into \eqref{Dex6-L1}, we obtain \eqref{Dex6-L1-1} and \eqref{Dex6-L1-2}, respectively.
\end{proof}

We now establish the following identities to confirm the modularity of the triples $\mathcal{DL}_1(A,B,C)$ in Table \ref{tab:MW-DL1}.
\begin{theorem}\label{T-MW-DL1-dual}
Let
\begin{align}
F(u,v,w,t;q):=\sum_{i,j,k,l\geq 0}\frac{u^{i}v^{j}w^{k}t^{l}q^{(i^2+j^2+k^2+3l^2)/2+il-jl}}{(q;q)_{i}(q;q)_{j}(q;q)_{k}(q;q)_{l}}.   
\end{align}
We have
\begin{align}
F(q^{\nu},q^{-\nu},q^{c},q^{\nu};q)&=\frac{(-q^{c+\frac{1}{2}};q)_{\infty}J_{5}\overline{J}_{\frac{1}{2}-\nu,1}}{J_{1}J_{1,5}}, \label{Dex6-L1-D1} \\
F(q^{\nu},q^{-\nu},q^{c},q^{\nu+1};q)&=\frac{(-q^{c+\frac{1}{2}};q)_{\infty}J_{5}\overline{J}_{\frac{1}{2}-\nu,1}}{J_{1}J_{2,5}}. \label{Dex6-L1-D2}
\end{align}
\end{theorem}

\begin{proof}
Summing over $k$ using \eqref{Euler2}, we have
\begin{align}\label{Dex6-L1-D}
F(u,v,w,t;q)=(-wq^{1/2};q)_{\infty}\sum_{i,j,l\geq 0}\frac{u^{i}v^{j}t^{l}q^{(i^2+j^2+3l^2)/2+il-jl}}{(q;q)_{i}(q;q)_{j}(q;q)_{l}}.    
\end{align}
Setting $\alpha=h=1$, $A_1=2$ and $B_1\in \{0,1\}$ in \eqref{Sq-start} (see also \cite[Theorem $4.3$]{Wang-rank3}), we obtain
\begin{align}\label{Z3-Ex2}
\sum_{i.j.k\geq 0}\frac{q^{(3i^2+j^2+k^2)/2+ij-ik+(\nu+B_{1})i+\nu j-\nu k+\nu^2/2}}{(q;q)_{i}(q;q)_{j}(q;q)_{k}} =\frac{J_{5}}{J_{1}J_{1+B_{1},5}}\sum_{n\in\mathbb{Z}+\nu}q^{n^2/2}.
\end{align}
Rewriting the variables $i,j,k$ as $\ell,i,j$, respectively, and then 
substituting the result into \eqref{Dex6-L1-D}, we obtain \eqref{Dex6-L1-D1} and \eqref{Dex6-L1-D2}.
\end{proof}

Finally, we give proofs for the modularity of the triples arising as lift-dual of the Milas--Wang example \eqref{MW-mt} through the lift $\mathcal{L}_3$.
\begin{proof}[Proof of Theorem \ref{thm-MW-L3}]
Summing over $i$ and $j$ using \eqref{Euler2}, we have
\begin{align}\label{Dex6-L3}
&F(u,v,w,t;q)=\sum_{k,l\geq 0}\frac{w^{k}t^{l}q^{2k^2+l^2+2kl}}{(q;q)_{k}(q;q)_{l}}(-uq^{\frac{1}{2}-(k+l)},-vq^{\frac{1}{2}+k};q)_{\infty} \nonumber\\ 
&=(-uq^{1/2},-vq^{1/2};q)_{\infty}\sum_{k,l\geq 0}\frac{u^{k+l}w^{k}t^{l}q^{k^2+(k+l)^2/2}(-q^{1/2}/u;q)_{k+l}}{(-vq^{1/2};q)_{k}(q;q)_{k}(q;q)_{l}} \nonumber\\ 
&=(-uq^{1/2},-vq^{1/2};q)_{\infty}\sum_{n=0}^{\infty}\sum_{k=0}^{n}\frac{(ut)^{n}(w/t)^{k}q^{k^2+n^2/2}(-q^{1/2}/u;q)_{n}}{(-vq^{1/2};q)_{k}(q;q)_{k}(q;q)_{n-k}}.
\end{align}

(1) Setting $u=v=w=t=1$ in \eqref{Dex6-L3}, we have
\begin{align}
&F(1,1,1,1;q)=(-q^{1/2};q)_{\infty}^{2}\sum_{n=0}^{\infty}\sum_{k=0}^{n}\frac{q^{k^2+n^2/2}(-q^{1/2};q)_{n}}{(-q^{1/2};q)_{k}(q;q)_{k}(q;q)_{n-k}}.
\end{align}    
Substituting the Bailey pair \eqref{H2} into \eqref{S1}, we obtain a new Bailey pair:
\begin{align}
\begin{split}
\alpha_{0}'(1;q)&=1,\quad \alpha_{n}'(1;q)=(-1)^nq^{2n^2}(q^{n/2}+q^{-n/2}),\\ 
\beta_{n}'(1;q)&=\sum_{k=0}^{n}\frac{q^{k^2}}{(-q^{1/2},q;q)_{k}(q;q)_{n-k}}.
\end{split}
\end{align}
Then substituting this Bailey pair into \eqref{S4*}, we deduce that
\begin{align}
&F(1,1,1,1;q)=(-q^{1/2};q)_{\infty}^{2}\sum_{n=0}^{\infty}q^{n^2/2}(-q^{1/2};q)_{n}\beta_{n}'(1;q) \nonumber\\ 
&=\frac{(-q^{1/2};q)_{\infty}^{3}}{(q;q)_{\infty}}\Big(1+\sum_{n=1}^{\infty}(-1)^{n}q^{5n^2/2}(q^{n/2}+q^{-n/2})\Big) \nonumber\\
&=\frac{(-q^{1/2};q)_{\infty}^{3}}{(q;q)_{\infty}}\sum_{n=-\infty}^{\infty}(-1)^{n}q^{(5n^2+n)/2}. 
\end{align}    
By \eqref{Jacobi} we obtain \eqref{thm-MW-id-1}.

(2) Setting $(u,v,w,t)=(q^{-1/2},q^{1/2},q^{2},q)$ in \eqref{Dex6-L3}, we have
\begin{align}
&F(q^{-1/2},q^{1/2},q^{2},q;q)=2(-q;q)_{\infty}^{2}\sum_{n=0}^{\infty}\sum_{k=0}^{n}\frac{q^{k^2+k+(n^2+n)/2}(-q;q)_{n}}{(q^2;q^2)_{k}(q;q)_{n-k}}.
\end{align}   
Substituting the Bailey pair \eqref{E4} into \eqref{S1}, we obtain a new Bailey pair:
\begin{align}
\begin{split}
&\alpha_{0}^{(1)}(1;q)=1,\\
&\alpha_{2n}^{(1)}(1;q)=q^{8n^2}(q^{2n}+q^{-2n}), \quad n\geq 1,\\
&\alpha_{2n+1}^{(1)}(1;q)=-q^{8n^2+6n+1}-q^{8n^2+10n+3}, \quad n\geq 0, \\ 
&\beta_{n}^{(1)}(1;q)=\sum_{k=0}^{n}\frac{q^{k^2+k}}{(q^2;q^2)_{k}(q;q)_{n-k}}.
\end{split}
\end{align}
Then substituting the Bailey pair $(\alpha_{n}^{(1)}(1;q),\beta_{n}^{(1)}(1;q))$ into \eqref{Lovejoy-a-aq-1}, we obtain the Bailey pair:
\begin{align}
\begin{split}
&\alpha_{0}^{(2)}(q;q)=1,\\
&\alpha_{2n}^{(2)}(q;q)=\frac{(1-q^{4n+1})q^{8n^2+2n}}{1-q}, \quad n\geq 1,\\
&\alpha_{2n+1}^{(2)}(q;q)=-\frac{(1-q^{4n+3})q^{8n^2+10n+3}}{1-q}, \quad n\geq 0,\\ 
&\beta_{n}^{(2)}(q;q)=\beta_{n}^{(1)}(1;q).
\end{split}
\end{align}
Finally, substituting this Bailey pair into \eqref{S2*}, we deduce that
\begin{align}
&F(q^{-1/2},q^{1/2},q^{2},q;q)=2(-q;q)_{\infty}^{2}\sum_{n=0}^{\infty}q^{(n^2+n)/2}(-q;q)_{n}\beta_{n}^{(2)}(q;q) \nonumber\\ 
&=\frac{2(-q;q)_{\infty}^{3}}{(q;q)_{\infty}}\Big(1-q+\sum_{n=1}^{\infty}(q^{10n^2+3n}-q^{10n^2+7n+1})-\sum_{n=1}^{\infty}(q^{10n^2-7n+1}-q^{10n^2-3n})\Big) \nonumber\\
&=\frac{2(-q;q)_{\infty}^{3}}{(q;q)_{\infty}}\Big(\sum_{n=-\infty}^{\infty}q^{10n^2-3n}-q\sum_{n=-\infty}^{\infty}q^{10n^2-7n}\Big). 
\end{align}    
By \eqref{Jacobi} we obtain \eqref{thm-MW-id-2}.

(3) Setting $(u,v,w,t)=(q^{-1},q,q^{2},q)$ in \eqref{Dex6-L3}, we have
\begin{align}
&F(q^{-1},q,q^{2},q;q)=q^{-1/2}(-q^{1/2};q)_{\infty}^{2}\sum_{n=0}^{\infty}\sum_{k=0}^{n}\frac{q^{k^2+k+n^2/2}(-q^{3/2};q)_{n}}{(-q^{3/2};q)_{k}(q;q)_{k}(q;q)_{n-k}}.
\end{align}   
Substituting the Bailey pair \eqref{BP-1} into \eqref{S1}, we obtain a new Bailey pair:
\begin{align}
\begin{split}
\alpha_{n}'(q;q)&=\frac{(1-q^{n+1/2})(-1)^{n}q^{2n^2+3n/2}}{1-q^{1/2}},\\ 
\beta_{n}'(q;q)&=\sum_{k=0}^{n}\frac{q^{k^2+k}}{(-q^{3/2},q;q)_{k}(q;q)_{n-k}}.
\end{split}
\end{align}   
Substituting this Bailey pair into \eqref{S4*}, we deduce that
\begin{align}
&F(q^{-1},q,q^{2},q;q)=q^{-1/2}(-q^{1/2};q)_{\infty}^{2}\sum_{n=0}^{\infty}q^{n^2/2}(-q^{3/2};q)_{n}\beta_{n}'(q;q) \nonumber\\ 
&=\frac{q^{-1/2}(-q^{1/2};q)_{\infty}^{3}}{(q;q)_{\infty}}\sum_{n=0}^{\infty}(1-q^{2n+1})(-1)^nq^{(5n^2+3n)/2} \nonumber\\
&=\frac{q^{-1/2}(-q^{1/2};q)_{\infty}^{3}}{(q;q)_{\infty}}\sum_{n=-\infty}^{\infty}(-1)^nq^{(5n^2+3n)/2}.
\end{align}    
By \eqref{Jacobi} we obtain \eqref{thm-MW-id-3}.

(4) Setting $(u,v,w,t)=(q^{1/2},q^{-1/2},1,1)$ in \eqref{Dex6-L3}, we have
\begin{align}
&F(q^{1/2},q^{-1/2},1,1;q)=2(-q;q)_{\infty}^{2}\sum_{n=0}^{\infty}\sum_{k=0}^{n}\frac{q^{k^2+(n^2+n)/2}(-1;q)_{n}}{(-1;q)_{k}(q;q)_{k}(q;q)_{n-k}}.
\end{align}    
Substituting the Bailey pair \eqref{BP-2} into \eqref{S1}, we obtain a new Bailey pair:
\begin{align}
\begin{split}\alpha_{0}'(1;q)&=1,\quad \alpha_{n}'(1;q)=\frac{(-1)^{n}}{2}(1+q^{n})^{2}q^{2n^2-n},\\ 
\beta_{n}'(1;q)&=\sum_{k=0}^{n}\frac{q^{k^2}}{(-1,q;q)_{k}(q;q)_{n-k}}.
\end{split}
\end{align}   
Substituting this Bailey pair into \eqref{S6*}, we deduce that
\begin{align}
&F(q^{1/2},q^{-1/2},1,1;q)=2(-q;q)_{\infty}^{2}\sum_{n=0}^{\infty}q^{(n^2+n)/2}(-1;q)_{n}\beta_{n}'(1;q) \nonumber\\ 
&=\frac{2(-q;q)_{\infty}^{3}}{(q;q)_{\infty}}\Big(1+\sum_{n=1}^{\infty}(1+q^{n})(-1)^nq^{(5n^2-n)/2}\Big) \nonumber\\
&=\frac{2(-q;q)_{\infty}^{3}}{(q;q)_{\infty}}\sum_{n=-\infty}^{\infty}(-1)^nq^{(5n^2-n)/2}.
\end{align}    
By \eqref{Jacobi} we obtain \eqref{thm-MW-id-4}.

(5) Setting $(u,v,w,t)=(q^{-1/2},q^{1/2},1,1)$ in \eqref{Dex6-L3}, we have
\begin{align}
&F(q^{-1/2},q^{1/2},1,1;q)=2(-q;q)_{\infty}^{2}\sum_{n=0}^{\infty}\sum_{k=0}^{n}\frac{q^{k^2+(n^2-n)/2}(-q;q)_{n}}{(q^2;q^2)_{k}(q;q)_{n-k}}.
\end{align}    
Substituting the Bailey pair \eqref{E1} into \eqref{S1}, we obtain a new Bailey pair:
\begin{align}
\begin{split}
\alpha_{0}'(1;q)&=1,\quad \alpha_{2n}'(1;q)=2q^{8n^2},\quad \alpha_{2n+1}'(1;q)=-2q^{8n^2+8n+2},\\ 
\beta_{n}'(1;q)&=\sum_{k=0}^{n}\frac{q^{k^2}}{(q^2;q^2)_{k}(q;q)_{n-k}}.
\end{split}
\end{align}   
Substituting this Bailey pair into \eqref{S5*}, we deduce that
\begin{align}
&F(q^{-1/2},q^{1/2},1,1;q)=2(-q;q)_{\infty}^{2}\sum_{n=0}^{\infty}q^{(n^2-n)/2}(-q;q)_{n}\beta_{n}'(1;q) \nonumber\\ 
&=\frac{4(-q;q)_{\infty}^{3}}{(q;q)_{\infty}}\Big(1+\sum_{n=1}^{\infty}(1+q^{2n})q^{10n^2-n}-\sum_{n=0}^{\infty}(1+q^{2n+1})q^{10n^2+9n+2}\Big) \nonumber\\
&=\frac{4(-q;q)_{\infty}^{3}}{(q;q)_{\infty}}\Big(\sum_{n=-\infty}^{\infty}q^{10n^2-n}-q^{2}\sum_{n=-\infty}^{\infty}q^{10n^2+9n}\Big). 
\end{align}    
By \eqref{Jacobi} we obtain \eqref{thm-MW-id-5}.

(6) Setting $(u,v,w,t)=(q^{-3/2},q^{3/2},q^{2},q)$ in \eqref{Dex6-L3}, we have
\begin{align}
&F(q^{-3/2},q^{3/2},q^{2},q;q)=2q^{-1}(-q;q)_{\infty}^{2}\sum_{n=0}^{\infty}\sum_{k=0}^{n}\frac{q^{k^2+k+(n^2-n)/2}(-q^2;q)_{n}}{(-q^2;q)_{k}(q;q)_{k}(q;q)_{n-k}}.
\end{align}    
Substituting the Bailey pair \eqref{BP-3} into \eqref{S1}, we obtain a new Bailey pair:
\begin{align}
\begin{split}
&\alpha_{0}^{(1)}(q;q)=1,\\
&\alpha_{n}^{(1)}(q;q)=\frac{2(1+q)(1-q^{2n+1})(-1)^{n}q^{2n^2+2n}}{(1-q)(1+q^{n})(1+q^{n+1})},\\ 
&\beta_{n}^{(1)}(q;q)=\sum_{k=0}^{n}\frac{q^{k^2+k}}{(-q^2,q;q)_{k}(q;q)_{n-k}}.
\end{split}
\end{align}
Substituting the Bailey pair $(\alpha_{n}^{(1)}(q;q),\beta_{n}^{(1)}(q;q))$ into \eqref{McLaughlin-a-a/q}, we obtain the Bailey pair:
\begin{align}
\begin{split}
&\alpha_{0}^{(2)}(1;q)=1,\\
&\alpha_{n}^{(2)}(1;q)=\frac{2(1+q)(-1)^{n}q^{2n^2-1}}{1+q^{n}}\Big(\frac{q^{2n+1}}{1+q^{n+1}}+\frac{1}{1+q^{n-1}}\Big),\\ 
&\beta_{n}^{(2)}(1;q)=\beta_{n}^{(1)}(q;q).
\end{split}
\end{align}

We have
\begin{align}\label{DEx6-L3-6}
&F(q^{-3/2},q^{3/2},q^{2},q;q)=2q^{-1}(-q;q)_{\infty}^{2}\sum_{n=0}^{\infty}q^{(n^2-n)/2}(-q^2;q)_{n}\beta_{n}^{(1)}(q;q) \nonumber\\ 
&=\frac{2q^{-1}(-q;q)_{\infty}^{2}}{1+q}\Big(\sum_{n=0}^{\infty}q^{(n^2-n)/2}(-q;q)_{n}\beta_{n}^{(2)}(1;q) \nonumber\\ 
&\quad \quad +\sum_{n=0}^{\infty}q^{1+(n^2+n)/2}(-q;q)_{n}\beta_{n}^{(1)}(q;q)\Big).
\end{align}

Substituting the Bailey pair $(\alpha_{n}^{(1)}(q;q),\beta_{n}^{(1)}(q;q))$ into \eqref{S2*}, we deduce that
\begin{align}\label{BP-3.1}
&\frac{q}{1+q}\sum_{n=0}^{\infty}q^{(n^2+n)/2}(-q;q)_{n}\beta_{n}^{(1)}(q;q) \nonumber\\
&=\frac{2q(-q;q)_{\infty}}{(q;q)_{\infty}}\sum_{n=0}^{\infty}\frac{(1-q^{2n+1})(-1)^nq^{(5n^2+5n)/2}}{(1+q^{n})(1+q^{n+1})} \nonumber\\
&=\frac{q(-q;q)_{\infty}}{(q;q)_{\infty}}\sum_{n=0}^{\infty}(-1)^nq^{(5n^2+5n)/2}\Big(\frac{1-q^{n+1}}{1+q^{n+1}}+\frac{1-q^{n}}{1+q^{n}}\Big) \nonumber\\ 
&=\frac{q(-q;q)_{\infty}}{(q;q)_{\infty}}\Big(\sum_{n=0}^{\infty}\frac{(-1)^nq^{(5n^2+5n)/2}}{1+q^{n+1}}-\sum_{n=0}^{\infty}\frac{(-1)^nq^{1+(5n^2+7n)/2}}{1+q^{n+1}} \nonumber\\
&\quad +\sum_{n=0}^{\infty}\frac{(-1)^nq^{(5n^2+5n)/2}}{1+q^{n}}-\sum_{n=0}^{\infty}\frac{(-1)^nq^{(5n^2+7n)/2}}{1+q^{n}}\Big).
\end{align}    
Substituting the Bailey pair $(\alpha_{n}^{(2)}(1;q),\beta_{n}^{(2)}(1;q))$ into \eqref{S5*}, we deduce that
\begin{align}\label{BP-3.2}
&\frac{1}{1+q}\sum_{n=0}^{\infty}q^{(n^2-n)/2}(-q;q)_{n}\beta_{n}^{(2)}(1;q) \nonumber\\
&=\frac{2(-q;q)_{\infty}}{(q;q)_{\infty}}\Big(\sum_{n=0}^{\infty}\frac{(-1)^nq^{(5n^2+3n)/2}}{1+q^{n+1}}+\sum_{n=1}^{\infty}\frac{(-1)^nq^{(5n^2-n)/2-1}}{1+q^{n-1}}\Big) \nonumber\\
&=\frac{2(-q;q)_{\infty}}{(q;q)_{\infty}}\Big(\sum_{n=0}^{\infty}\frac{(-1)^nq^{(5n^2+3n)/2}}{1+q^{n+1}}-\sum_{n=0}^{\infty}\frac{(-1)^nq^{(5n^2+9n)/2+1}}{1+q^{n}}\Big).
\end{align}

Finally, substituting \eqref{BP-3.1} and \eqref{BP-3.2} into \eqref{DEx6-L3-6}, we have
\begin{align}
&F(q^{-3/2},q^{3/2},q^{2},q;q) \nonumber\\ 
&=2q^{-1}\frac{(-q;q)_{\infty}^{3}}{(q;q)_{\infty}}\Big(\sum_{n=0}^{\infty}(-1)^nq^{(5n^2+3n)/2}+\sum_{n=0}^{\infty}(1-q^{n+1})(-1)^nq^{(5n^2+3n)/2} \nonumber\\
&\quad +\sum_{n=0}^{\infty}(1-q^{n})(-1)^nq^{1+(5n^2+5n)/2}-\sum_{n=0}^{\infty}(-1)^nq^{1+(5n^2+7n)/2}\Big) \nonumber\\
&=4q^{-1}\frac{(-q;q)_{\infty}^{3}}{(q;q)_{\infty}}\Big(\sum_{n=0}^{\infty}(-1)^nq^{(5n^2+3n)/2}-\sum_{n=0}^{\infty}(-1)^nq^{1+(5n^2+7n)/2}\Big) \nonumber\\
&=4q^{-1}\frac{(-q;q)_{\infty}^{3}}{(q;q)_{\infty}}\sum_{n=-\infty}^{\infty}(-1)^nq^{(5n^2+3n)/2}. 
\end{align}
By \eqref{Jacobi} we obtain \eqref{thm-MW-id-6}.
\end{proof}

\subsection*{Acknowledgements}
This work was supported by the National Key R\&D Program of China (Grant No.\ 2024YFA1014500).

\end{document}